\documentclass[12pt]{article}
\usepackage{amsmath}
\usepackage{amssymb}
\usepackage[english]{babel}
\usepackage{amsthm}
\usepackage{graphicx}
\usepackage{tikz}
\usepackage{enumerate}
\usepackage{faktor}

\newcommand \ran {{\mathrm {ran}}}

\newcommand \Z  {{\mathbb Z}}

\newcommand \MI {{M^{(I)}\!}}

\newtheorem{theorem}{Theorem}
\newtheorem{proposition}[theorem]{Proposition}
\newtheorem{lemma}[theorem]{Lemma}

\newtheorem{fact}[theorem]{Fact}
\newtheorem{claim}{Claim}[theorem]
\newtheorem*{claim*}{Claim}
\theoremstyle{definition}
\newtheorem{definition}[theorem]{Definition}
\newtheorem{example}[theorem]{Example}

\theoremstyle{definition}
\newtheorem{remark}[theorem]{Remark}

\title{Elementary extensions of almost o-minimal structures}
\author{Mourad Berraho\thanks{b.mourad87@hotmail.com} \and 
Akito Tsuboi\thanks{tsuboi@math.tsukuba.ac.jp}}
\date{}
\begin{document}
\maketitle


\begin{abstract}
    This paper investigates almost o-minimal structures, a weakening of o-minimality introduced by Fujita to capture structures that lie outside the classical o-minimal framework.  
In contrast to o-minimality and local o-minimality, almost o-minimality is not preserved under elementary equivalence.  
This raises the natural question of whether every almost o-minimal structure admits a proper elementary extension that is again almost o-minimal.  
The main result of this paper provides an affirmative answer to this question.
\end{abstract}

\section{Introduction}
The notion of o-minimality has revealed a deep connection between logic and real algebraic geometry. However, simple structures such as $(\mathbb{R},+,0, \Z)$ do not fall within this framework. To accommodate such examples, several weaker notions have been introduced that retain some desirable properties of o-minimality while encompassing a broader class of structures. These include locally o‑minimality, definable completeness (DC), and type completeness (TC), which have been studied extensively in \cite{FKK2022}, \cite{kawa2012}, and \cite{Schoutens2014}.
The definition of definable completeness (DC in short) for a structure $M$ is equivalent to the following condition: every bounded definable subset $D\subset M$ has both a supremum and an infimum in $M$. This means that whenever a definable set does not “escape to infinity,” its endpoints are already controlled inside the structure. In particular, any bounded convex definable set must be an interval whose endpoints lie in $M$. This is exactly the kind of tameness one expects in an o‑minimal structure. Thus, in a sense, definable completeness captures the part of o‑minimality that concerns the behavior of definable sets on bounded domains. So, informally, the condition DC asserts that bounded regions resemble an o‑minimal structure. On the other hand, Schoutens introduced the theory DCTC, which generalizes locally o‑minimal expansions of ordered fields (see \cite{Schoutens2014}). 
Here, the additional condition TC (type completeness), applied to the points “at infinity,” ensures that the behavior of definable sets eventually resembles that of an o‑minimal structure. 
More precisely, for any definable set $X\subseteq M$ in a model satisfying TC, there exist $a,b\in M$ such that
$X\cap \{ x<a\} \quad \mathrm{and}\quad X\cap \{ x>b\}$
are either empty or open intervals.
Fujita proposed a different weakening of o‑minimality, introducing the notion of almost o‑minimality. Every locally o‑minimal expansion of $(\mathbb{R},<)$ is almost o‑minimal, but this property is not preserved under elementary equivalence (see \cite{M. Fujita 2022}, Proposition 4.14). At the SAML 2024 Kyoto meeting \cite{4}, Fujita asked the following question:
\begin{itemize}
    \item Does every almost o-minimal structure admit a proper elementary extension?
\end{itemize}
This note provides an affirmative answer to that question.  
The proof is obtained by combining ideas from the proof of Gaifman’s theorem for Peano arithmetic  
with an appropriate variant of Łoś’s theorem.
Our main results contain the following statement:
\medbreak\noindent
{\bf Theorem:} 
Let $M=(M,<,\dots)$ be an almost o-minimal structure. 
Then, there is a proper elementary extension $N \succ M$ that is still almost o-minimal.
\medbreak
The reader is assumed to have a basic knowledge of model theory, including the compactness theorem and the construction of ultraproducts.  
We explain the basic facts about o‑minimality and related notions so that readers unfamiliar with them can follow our arguments, although some acquaintance with these topics may be helpful for understanding the background.

\section{Preliminaries}
Our notation is largely standard.  
The set of non‑negative integers is denoted by \(\omega\), and we use \(m,n,\dots\) to refer to elements of \(\omega\).  
When working in a fixed structure \(M\), we write \(a,b,\dots\) for elements of \(M\), and \(\bar a, \bar b,\dots\) for finite tuples from \(M\).  
By a slight abuse of notation, for a finite tuple \(\bar a = (a_0,\dots,a_{n-1})\), we sometimes write \(\bar a \in M\) to mean that \(a_i \in M\) for all \(i < n\).
Subsets of \(M\) are denoted by capital letters such as \(A,B\).
We use \(A \subset B\) to denote (not necessarily proper) inclusion.
The cardinality of $A$ is denoted by $|A|$. 
For two structures \(M\) and \(N\), we write \(M \prec N\) (or $N \succ M$) to mean that \(M\) is an elementary substructure of \(N\).  
The inclusion is not necessarily proper. 

Formulas are denoted by Greek letters such as \(\varphi, \psi\), and variables by \(x, y, \dots\).  
When the free variables of a formula \(\varphi\) are contained in a tuple \(\bar x\), we often write \(\varphi(\bar x)\).
A subset of \(M\) is called definable if it is of the form
\[
    \{\bar a \in M^n : M \models \varphi(\bar a, \bar b)\}, 
\]
for some formula \(\varphi(\bar x, \bar y)\) and some parameters \(\bar b \in M\).  
The above definable set is often written as $\varphi(M^n,\bar b)$. 
\medbreak
Let $M=(M,<,\dots)$ be a structure expanding a dense linear ordering $(M,<)$ without endpoints. 
\begin{definition}
    \begin{enumerate}
        \item A structure \(M\) is called \emph{o-minimal} if every definable (with parameters) subset \(D \subset M\) is a finite union of (open) intervals and points.
        
        \item A structure \(M\) is called \emph{locally o-minimal} if for every definable set \(D \subset M\) and every \(a \in M\), there exists an interval \(I \ni a\) such that \(D \cap I\) is a finite union of points and intervals.  
In particular, for a sufficiently small left interval \((a_0,a)\), the set \(D \cap (a_0,a)\) is either \((a_0,a)\) or \(\emptyset\); the same holds for a sufficiently small right interval.
        \item A structure \(M\) is called \emph{definably complete} if every bounded definable set \(B \subset M\) has both \(\sup B\) and \(\inf B\) in \(M\).  
    \end{enumerate}
\end{definition}

\begin{remark}
\begin{enumerate}
\item Clearly, every o-minimal structure is both locally o-minimal and definably complete.
\item The property of being o-minimal is elementary; that is, if $M \equiv N$ and $M$ is o-minimal, then $N$ is also o-minimal (see \cite{KPS86}). 
In the proof, the uniform finiteness theorem (Fact \ref{unif-finite}) plays an essential role (see \cite{vdDries1998}).
\item Local o-minimnality and definable completeness are also elementary properties. These facts are relatively straightforward to verify.
For example, in the case of local o‑minimality, elementarity can be verified directly from the “in particular” clause in the definition.
\end{enumerate}
\end{remark}

\begin{fact}[Uniform Finiteness]\label{unif-finite}
Let \(M\) be an o-minimal structure, and let \(D \subset M^{n+1}\) be a definable set satisfying the following property:
\begin{itemize}
    \item[(*)] For every \(\bar a \in M^n\), the fiber
$D_{\bar a} := \{\, b \in M : (\bar a, b) \in D \,\}$  is finite. 
\end{itemize}
Then there exists a uniform bound \(k \in \omega\) such that
$|D_{\bar a}| < k$ for all $\bar a \in M^n$.
\end{fact}

\begin{fact}[\cite{FKK2022}, \cite{KT}] \label{pairwise-conti}
Let $M$ be a locally o-minimal structure with definable completeness, and let $f: M \to M$ be a definable function. Then there exists a definable discrete set $D \subset M$ with the following property: For every pair of consecutive elements $a < b$ in $D \cup \{-\infty, \infty\}$, the restriction $f|_{(a, b)}$ is continuous and is either constant or strictly monotone.
\end{fact}

The following definition is due to Fujita. 
\begin{definition}[\cite{4}]
A structure $M=(M,<,\dots)$ is called almost o-minimal, if all bounded definable subsets of $M$ can be written as a finite union of points and intervals. 
\end{definition}

\begin{example}
Let $M = (\mathbb{Q}, <, \mathbb{Z})$, where $\mathbb{Z}$ is interpreted as a unary predicate. In this case, $M$ is clearly an almost o-minimal structure. Now, let $N$ be the elementary extension obtained by taking two copies of $M$ ordered linearly; more precisely, let the domain of $N$ be $\{0, 1\} \times \mathbb{Q}$ equipped with the lexicographic order, and let $\mathbb{Z}^N = \{0, 1\} \times \mathbb{Z}$. While $M \equiv N$, the structure $N$ is no longer almost o-minimal. This is because the bounded interval between $(0, 0)$ and $(1, 0)$ contains infinitely many $\mathbb{Z}$-elements, violating the requirements of almost o-minimality.
\end{example}

\begin{example}
Let \(M\) be the structure \((\mathbb{Q}\setminus\{0\},<,U)\), where  
\(U = \{\pm 1/n : n \in \omega \setminus \{0\}\}\) is given by a unary predicate.  
Then \(M\) is a model of DCTC, but it is not almost o-minimal.
\noindent\emph{Proof.}
The theory \(Th(M)\) has quantifier elimination after expanding the language to  
\(\{<, U,-1,1, D_n(x,y) \ (n \in \omega)\}\), where \(D_n(x,y)\) expresses that there are at least \(n\) many \(U\)-points in the interval \((x,y)\).  
This can be shown by a standard back-and-forth argument.  
Quantifier elimination in this expanded language implies that \(M\) satisfies DCTC, while the discrete nature of \(U\) shows that \(M\) is not almost o-minimal.
\end{example}

\begin{remark}
By definition, it is immediate that an almost o-minimal structure $M$ is both locally o-minimal and definably complete. 
In particular, if $N \equiv M$ then $N$ is locally o-minimal and definably complete. 
\end{remark}

Two main ingredients in the proof of our main results are the notion of bounded ultrapowers from set theory and Gaifman's theorem for Peano arithmetic.  
Neither of these tools is widely familiar among researchers in o-minimality, so we briefly recall the basic facts surrounding them.
In model theory, the ultrapower is a powerful construction that produces a new structure from a given structure \(M\).  
It is obtained by taking an infinite direct product of \(M\) and identifying elements that are equivalent modulo a non-principal ultrafilter.  
However, this construction does not preserve almost o-minimality in the sense of Fujita.  
This phenomenon is analogous to the fact that a non-principal ultrapower of the structure $(V_\omega, \in)$ of hereditary finite sets is not well-founded (see \cite{Davis1977}). 
To address this issue within the context of o-minimal-like structures, we introduce the notion of a \emph{bounded ultrapower}, which parallels similar constructions in set theory and non-standard analysis.

In general, a bounded ultrapower does not preserve the truth of all formulas, but only of those with bounded quantifiers.  
Our goal, however, is to construct a proper elementary extension of an almost o-minimal structure.  
To overcome this obstacle, we adapt Gaifman's argument for models of Peano arithmetic to our setting.
Broadly speaking, Gaifman’s Splitting Theorem states that any elementary extension
\(N\) of a model \(M\) of Peano Arithmetic (PA) can be decomposed into an elementary
cofinal extension \(M' \succ M\), followed by an elementary end extension \(N \succ M'\) (see \cite{Gaifman} and \cite{{Kaye1991}}).
Since our bounded ultrapower is a cofinal extension of the original model, this suggests
that Gaifman’s argument may be adapted to our setting—and indeed, it can.


\section{Bounded ultrapowers}
Let $L$ be a language containing the less-than symbol $<$, and let $M = (M, <, \dots)$ be an $L$-structure, where $<$ is a dense linear order without endpoints. 
Let $U$ be a non-principal ultrafilter on the index set $I$. We denote
by $M^{(I)}$ the set of all bounded functions from $I$ to $M$. 
Here, $f$ is considered bounded if there exist elements $a, b \in M$ such that $a < f(i) < b$ holds for all $i \in I$. 
On the set $M^{(I)}$,  we define an equivalence relation $\sim$ by:
\[
f \sim g \iff \{i \in I : f(i) = g(i)\} \in U.
\]
Let $[f]$ denote the equivalence class of $f$, and let $M^{(I)}/U$ denote the set of all equivalence classes. 
We consider $M^{(I)}/U$ as an $L$-structure by defining the interpretation of symbols in $L$:
\[
M^{(I)} \models R([f], . . .) \iff \{i \in I : M \models R(f(i), . . .)\} \in U,
\]
where $R$ is a predicate symbol. Function symbols are also interpreted in
natural way. $M^{(I)}/U$ will be called the bounded ultrapower of $M$ modulo $U$.
For simplicity, we will denote $M^{(I)}/U$ by $M^*$.

\begin{definition}
A quantifier of the form \(\forall x \in D\) or \(\exists x \in D\) is called a \emph{bounded quantifier}, where \(D\) is a bounded interval in \(M\).  
Let \(\varphi\) be an \(L\)-formula.  
If every quantifier occurring in \(\varphi\) is bounded, then \(\varphi\) is called a \(\Delta_0\)-formula (or a \emph{bounded formula}).  
The classes \(\Pi_n\)- and \(\Sigma_n\)-formulas are defined in the usual way.  
In particular, a formula of the form
\(
    \forall x\, \exists y\, \forall z\, \theta(x,y,z)
\)
is a \(\Pi_3\)-formula, provided that \(\theta\) is bounded.
\end{definition}

The following lemma, which is an analogue of \L\'os's lemma, can be established easily by induction on the number of bounded quantifiers.
We provide only a sketch of the proof here.  
For further details, the reader may consult, for example, Davis’s book \cite{Davis1977}.
Recall that $M^*$ denotes $M^{(I)}/U$.
\begin{lemma}\label{Los}
For all $\Delta_0$-formulas $\varphi (\bar x)$, $\bar x = x_1, \dots, x_m$ and $\bar f = f_1, . . . , f_m \in M^{(I)}$, 
\[
M^*
\models \varphi ([ \bar f]) \iff  \{i \in  I : M \models \varphi (f_1(i), . . . , f_m(i))\} \in  U,
\]
where $[\bar f] = [f_1], . . . , [f_m]$ and $\bar f(i) = f_1(i), \dots , f_m(i)$.
\end{lemma}
\begin{proof}
We prove the lemma by induction on the number of bounded quantifiers occurring in \(\varphi\).  
The base case is immediate, as it follows directly from the definition of \(M^*\).
 The crucial induction step is when $\varphi$  has the form 
 \[
 \exists y \in  (x_1, x_2) \big(\theta (y, \bar x)\big),
 \]
 where $y \in  (x_1, x_2)$ means that $y$ belongs to the open interval $(x_1,x_2)$. 
 Suppose $M^*\models \exists y \in  ([f_1], [f_2])\big(\theta (y, [
\bar f])\big)$ holds.
Then, there is $h \in  M^{(I)}$ such that $M^*
\models [h] \in  ([f_1], [f_2]) \wedge \theta ([h], [
\bar f])$ holds. 
So, by the induction hypothesis, we have
\[
\{i \in  I : M \models h(i) \in  (f_1(i), f_2(i)) \wedge \theta (h(i),\bar f(i))\} \in  U.
\]
Hence,
$\{i \in  I : M \models \exists y \in  (f_1(i), f_2(i))\big(\theta (y, \bar f(i)\big)\}$ also belongs to $U$.
For the other implication, suppose $\{i \in  I : M \models \exists y \in 
(f_1(i), f_2(i))\big(\theta (y, \bar f(i)\big)\} \in  U$. 
Then, for each $i$, choose $h(i) \in M$ such
that $\{i : h(i) \in  (f_1(i), f_2(i)) \wedge \theta (y, \bar f(i))\} \in  U$. 
Since $f_1$ and $f_2$ are bounded, $h:I \to M$ is a bounded function. 
By the induction
hypothesis, $M^* \models [h] \in  ([f_1], [f_2]) \wedge \theta ([h], [\bar f])$ holds. 
So, we have $M^* \models \exists y \in  ([f_1], [f_2])\big(\theta (y, [\bar f])\big)$ as desired.
\end{proof}
For $a \in  M$, a can be considered the function in $M^{(I)}$ that constantly yields the value $a$. By identifying $a$ and $[a]$, it is easy to see that $M$ is coinitial and cofinal in $M^*$. 
By the above equivalence in Lemma \ref{Los}, we also have $M \prec_{\Delta_0} M^*$, namely $M^*$ is a $\Delta_0$-elementary extension of $M$. 
We now consider the following condition, which corresponds to a weaker form of the replacement schemes.  
A formula with parameters from \(M\) is called an \(L(M)\)-formula.  
For a tuple \(\bar y = (y_0,\dots,y_{n-1})\), the quantification
\(
    \forall \bar y \in (a,b)
\)
is shorthand for
\[
    \forall y_0 \in (a,b)\;\dots\;\forall y_{n-1} \in (a,b).
\]
\begin{itemize}
    \item[($\dagger$)]  For all $L(M)$-formulas $\varphi (x, \bar y)$ and $a, b \in  M$, if $M \models \forall \bar y \in 
(a, b)\exists x(\varphi (x, \bar y))$, then there exist $d, e \in  M$ such that $M \models \forall \bar y \in 
(a, b)\exists x \in  (d, e)(\varphi (x, \bar y))$.
\end{itemize}

The following theorem is an analogue of Gaifman's Splitting Theorem for Peano arithmetic (see, e.g., \cite{Kaye1991}).  
Our version is formulated in the continuous setting.
\begin{theorem}\label{thm-elementary}
    Assume that $M$ satisfies the condition $(\dagger)$. Then, $M \prec M^*$, and $M$ is cofinal and coinitial in $M^*$. 
\end{theorem}
    \begin{proof}
    We prove the following claim by induction on the complexity of formulas:
    \begin{claim}
        For all $L(M)$-formulas $\varphi(\bar y)$, and for all $\bar f=f_1,\dots,f_m \in \MI$,
    \[
    M^* \models \varphi([\bar f]) \iff \{i \in I: M\models \varphi(\bar f(i))\} \in U.
    \]
    \end{claim}
    The crucial step is again the case when $\varphi(\bar y)$ has the form $\exists x(\theta(x, \bar y))$. 
    By the definition of $M^*$, we have the equivalence: 
\[
         M^* \models \exists x (\theta(x,[\bar f])) \iff M^* \models  \theta([g],[\bar f]) \text{ for some $g \in \MI$}. 
\]
By the induction hypothesis, for all $h \in \MI$,
\[\tag{*}
 M^* \models  \theta([h],[\bar f]) \iff 
 \{i \in I: M \models \theta(h(i), \bar f(i))\} \in U 
\]
Since $ M \models \theta(h(i), \bar f(i))$ implies $ M \models \exists x (\theta(x,\bar f(i)))$, we have 
\[
         M^* \models \exists x (\theta(x,[\bar f])) \Longrightarrow  \{i \in I: M \models \exists x(\theta(x, \bar f(i)))\} \in U. 
\]
Now we show the converse of the above. So, suppose that $A:=\{i \in I: M \models \exists x(\theta(x, \bar f(i)))\} \in U$. 
As $\bar f$ are bounded functions, there exist $a,b \in M$ such that $a<\bar f(i) <b$ holds for all $i \in I$.
Observe that the sentence $\forall \bar y \in (a,b) \exists x (\exists z (\theta(z,\bar y)) \to \theta(x,\bar y))$ is true in $M$. 
Hence, by the condition $(\dagger)$, we have $d,e \in M$ such that 
\[
M \models \forall \bar y \in (a,b) \,\exists x \in (d,e)\big(\exists z (\theta(z,\bar y)) \,\to\, \theta(x,\bar y)\big).
\]
For $i \in A$, since $\exists z \theta(z,\bar f(i))$ is true, we have 
\[
M \models \exists x \in (d,e) \big(\theta(x, \bar f(i))\big). 
\]
Thus, we can choose a function $h \in \MI$ such that $M \models \theta(h(i), \bar f(i))$ holds for all $i \in A$. 
From this and the induction hypothesis (*), we have $M^*\models \theta([h],[\bar f])$, and hence $M^* \models \exists x (\theta(x,[\bar f]))$. 
Thus, the converse has also been proven. 
\end{proof}

\section{Some generalizations}
As in the previous section, $M=(M,<,\dots)$ and $U$ is a non-principal ultrafilter on $I$.  
We introduce the following two sets:
\begin{itemize}
    \item $M^{(I)^+}$ denotes the set of all functions $f:I \to M$ such that $f$ is upper bounded; 
    \item $M^{(I)^-}$ denotes the set of all functions $f:I \to M$ such that $f$ is lower bounded. 
\end{itemize}
Similar to the case of bounded ultrapowers, we can also introduce new structures $M^{(I)^+}\!/U=\{[f]: f \in M^{(I)^+}\}$ and $M^{(I)^-}\!/U=\{[f]: f \in M^{(I)^-}\}$. 
These will be called an upper-bounded ultrapower and a lower-bounded ultrapower, respectively. 
Finally, we define two conditions: 
\begin{itemize}
    \item[$(\dagger)^+$] 
     For all $L(M)$-formulas $\varphi(x, \bar y)$ and $a \in M$, if $M \models \forall \bar y <a \exists x (\varphi(x,y))$, then there exists $d \in M$ such that $M \models \forall \bar y < a \exists x <d (\varphi(x,y))$.  
   \item [$(\dagger)^-$] For all $L(M)$-formulas $\varphi(x, \bar y)$ and $a \in M$, if $M \models \forall \bar y > a \exists x (\varphi(x,y))$, then there exists $d \in M$ such that $M \models \forall \bar y > a \exists x > d (\varphi(x,y))$.  
\end{itemize}

\begin{proposition}
\begin{enumerate}\label{prop-bounded}
    \item     If $M$ satisfies condition $(\dagger)^+$, then $M \prec M^{(I)^+}\!\!/U$. 
    \item    If $M$ satisfies condition $(\dagger)^-$, then $M \prec M^{(I)^-}\!\!/U$.   
\end{enumerate}
\end{proposition}
\begin{proof}
    We can adapt the proof of Theorem~\ref{thm-elementary} with only minor modifications.  
For instance, in the proof of item~1, we work with quantifiers of the form  
\(\forall x \in (-\infty,a)\) and \(\exists x \in (-\infty,a)\) in place of bounded quantifiers.  
With this change, the argument proceeds in parallel to that of Lemma \ref{Los} and Theorem \ref{thm-elementary}.
\end{proof}

\section{Almost o-minimal structures}
    If $M$ is almost o-minimal, then the theory $Th(M)$ is locally o-minimal with definable completeness.

\begin{remark}\label{rem-interval-structure}
    Let $M$ be almost o-minimal. For elements $a<b$ in $M$, let $L^*=L \cup \{P_D\}_D$, where each $P_D$ is a predicate symbol, and $D$ runs over the $M$-definable subsets of $(a,b)^m$ with $m \in \omega$.  The interval $(a,b)$ is considered as an $L^*$-structure, by the natural interpretation, i.e., $P_D{}^{(a,b)}=D$. 
    Then, the $L^*$-structure $(a,b)$ is o-minimal, and all definable sets in $(a,b)$ are definable in $M$. 
\end{remark}

The following are auxiliary definitions that will be used in proving the main results of this section.
\begin{definition}
A structure $M$ is referred to as `downward o-minimal' (`upward o-minimal') if all its definable subsets that are bounded above (or below, respectively) can be expressed as a finite union of points and intervals.
\end{definition}
It is clear that if \( M \) is downward (or upward) o-minimal, then \( M \) is almost o-minimal; in particular, all models $N \equiv M$ are locally o-minimal and definably complete. 
By an argument similar to that in Remark~\ref{rem-interval-structure}, if \( M \) is downward o-minimal,  
then the structure \( (-\infty,a) \), equipped with the augmented predicates for definable sets, is o-minimal.  
The same conclusion holds in the case of upward o-minimality.

\begin{remark}\label{remark-3-cases}
   Suppose that a structure $M$ is almost o-minimal, but not o-minimal. 
   Then, one of the following must hold.
   \begin{enumerate}
       \item $M$ is downward o-minimal but not upward o-minimal;
       \item $M$ is upward o-minimal but not downward o-minimal;
       \item $M$ is neither downward nor upward o-minimal. 
   \end{enumerate}
   \begin{proof}
       Suppose 3 is not the case. So, $M$ is downward o-minimal, or upward o-minimal. By symmetry, we can assume $M$ is downward o-minimal.
       If $M$ were also upward o-minimal, then $M$ would satisfy the definition of o-minimality, contradicting the initial assumption on $M$. 
       Thus, item 1 holds in this case.  
   \end{proof}
\end{remark}

\begin{theorem}\label{thm-almost-o-mini}
    Let $M$ be an almost o-minimal structure that is not o-minimal. 
    \begin{enumerate}
        \item Suppose that $M$ is downward o-minimal. 
        Let $N$ be a cofinal elementary extension of $M$. 
        Then, $N$ is downward o-minimal. 
        \item Suppose that $M$ is upward o-minimal. 
        Let $N$ be a coinitial elementary extension of $M$. 
        Then, $N$ is upward o-minimal. 
        \item Suppose that $M$ is neither downward nor upward o-minimal. 
        Let $N$ be a cofinal and coinitial elementary extension of $M$. Then, $N$ is almost o-minimal (and neither downward nor upward o-minimal). 
    \end{enumerate}
\end{theorem}
     \begin{proof}
     Since the other cases are proved similarly, we prove 1. 
     For a contradiction, we assume that $N$ is not downward o-minimal, and let $D \subset N$ be its witness. 
     Since $N$ is locally o-minimal, by replacing $D$ with its frontier $\partial D$, we may assume that $D$ is an infinite definable discrete set that is bounded above.
     (This kind of reduction will also be used in the proof of Theorem~\ref{main}.)
     Since $M$ is cofinal in $N$, we can choose $a \in M$ such that $D <a$. 
     So, there is a non-algebraic formula $\varphi(x,\bar d)$ with $\bar d \in N$  such that (i) $N \models \forall x (\varphi(x,\bar d) \to x<a)$ and (ii) $\varphi(x, \bar d)$ defines a discrete set in $N$. 
Here, we can assume that $\bar{d} < a$ holds in $N$, by replacing $a$ with a larger element of $M$ if necessary.
Using \( \varphi \), we can easily obtain another formula 
\(\psi(x,a,\bar y)\), which defines a discrete set for all \(\bar y\), 
together with tuples \(\bar d_n \in N\) satisfying \(\bar d_n < a\) 
for \( n \in \omega \), such that
\begin{itemize}
  \item[(*)] \(\psi(N,a,\bar d_n)\) is a discrete subset of the interval 
  \((-\infty,a)^N\) of cardinality at least \(n\), for all \(n \in \omega\).
\end{itemize}
         From $M \prec N$ , using $(*)$, we get $\bar e_n \in M$ with $\bar e_n<a$ $(n \in \omega)$ such that each $\psi(M,a,\bar e_n)$ is a discrete (hence, finite) subset of $(-\infty,a)^M$ of cardinality $\geq n$.
This contradicts the o-minimality of $(-\infty,a)^M$ (see Remark \ref{rem-interval-structure}). (Here we used the fact that o-minimal structure has uniform finiteness property. See Fact \ref{unif-finite}.) 
 \end{proof}

\begin{theorem}\label{main}
    Let $M$ be almost o-minimal. 
    Then, there is a proper elementary extension $M^* \succ M$ that is still almost o-minimal. 
\end{theorem}
\begin{proof}
If $M$ is o-minimal, then such an extension $M^*$ certainly exists, as o-minimality is preserved under elementary extensions. 
So, we can assume $M$ is not o-minimal, and one of the cases 1--3 in Remark \ref{remark-3-cases} is true.
Cases 1 and 2 can be treated similarly, so we assume case 1 holds; that is, $M$ is downward o-minimal but not upward o-minimal. 
As $M$ is not upward o-minimal, there exists a definable discrete set $D \subset M$ that is not upper-bounded.
\begin{claim}
$(\dagger)^+$ is satisfied in $M$.   
\end{claim}
Let $a \in M$ be arbitrary. 
First we prove the following statement by induction on $n$:
\begin{itemize}
    \item[$(*)_n$] Let $F:((-\infty,a)^M)^n \to D$ be a definable function. Then, $\ran(F)$ is upper-bounded. 
\end{itemize}
For the base case, let \(F : (-\infty,a)^M \to D\) be definable.  
Let \(X \subset (-\infty,a)^M \) be the set of discontinuities of \(F\).
By Fact \ref{pairwise-conti}, $X$ is a definable discrete set in $M$. 
Since $X$ is also definable in the o-minimal structure \( (-\infty,a)^M \) with augmented predicates (see Remark~\ref{rem-interval-structure}), $X$ is a finite set. 
Because \( D \) is discrete, on each interval where \( F \) is continuous,  
the function \( F \) must be constant.  
Consequently, \(\operatorname{ran}(F)\) is a finite set, and bounded above.
We now consider the case \( F : ((-\infty,a)^M)^{\,n+1} \to D \).  
For each \( n \)-tuple \( \bar b \in ((-\infty,a)^M)^n \), define
\[
   F_{\bar b}(y) = F(\bar b,y).
\]
Since \( F_{\bar b} \) maps \( (-\infty,a)^M \) into \( D \), the base case applies.  
Define
\[
   G(\bar b) = \max \bigl(\operatorname{ran}(F_{\bar b})\bigr).
\]
Then \( G : ((-\infty,a)^M)^n \to D \) is a definable function.  
By the induction hypothesis, \(\operatorname{ran}(G)\) is bounded above.  
It follows that \( \ran(F) \) is also bounded above.  
Hence, \((*)_n\) holds for all \( n \).

To derive a contradiction, we assume that $(\dagger)^+$ does not hold. 
Then, there is a formula $\varphi(x,\bar y)$ and $a \in M$ with the following properties:
\begin{enumerate}
    \item  $M \models \forall \bar y < a \exists x (\varphi(x, \bar y))$; 
    \item For all $d \in M$, $M \models \exists \bar y <a \forall x<d (\neg \varphi(x, \bar y))$.
\end{enumerate}
Let $n=|\bar y|$. For each $\bar b <a$, define 
\[
   F(\bar b) = \min \{\, w \in D : M \models \exists x < w \, (\varphi(x,\bar b)) \,\}.
\]
Such a minimum element $w \in D$ exists, since $M$ is downward o-minimal, and $D$ is discrete. 
By Property 2, $\ran(F)$ must be unbounded. However, this contradicts $(*)_n$ above. (End of Proof of Claim A).
\medbreak
For case 3, we assume that \(M\) is merely almost o-minimal; that is, \(M\) is almost o-minimal,  
 but neither downward o-minimal nor upward o-minimal.  
 Consequently, there exists a discrete definable set \(D \subset M\) such that \(D\) is unbounded both above and below.
We claim the following in this case. 
\begin{claim}\label{claim-b}
$(\dagger)$ is satisfied in $M$.   
\end{claim}

Let $a<b \in M$ be arbitrary. 
 We prove the following statement. 
 \begin{itemize}
 \item[$(**)_n$] Let $F:(a,b)^n \to D$ be a definable function.
 Then, $\ran(F)$ is a bounded set. 
 \end{itemize}
 $(n=1)$: $F:(a,b) \to D$. 
 Since \(M\) is locally o-minimal and definably complete,  
 the set of points of discontinuity of \(F\) is a discrete subset of $(a,b)$.  
 By almost o-minimality, this set is in fact finite.  
 On each interval where \(F\) is continuous, the discreteness of \(D\) implies that \(F\) must be constant on that interval.
 Thus, $\ran(F)$ is finite, and hence is a bounded set. 

 $(n+1)$: 
 Let \(F : (a,b)^{n+1} \to D\).  
 For each \(\bar c \in (a,b)^n\), define the function
 \[
   F_{\bar c} : (a,b) \to D, \quad F_{\bar c}(x) := F(\bar c, x).
\]
 By the base case, \(\operatorname{ran}(F_{\bar c})\) is finite.  
 Hence we may define definable functions \(G_0, G_1 : (a,b)^n \to D\) by
 \[
    G_0(\bar c) := \min(\operatorname{ran}(F_{\bar c})), \qquad
    G_1(\bar c) := \max(\operatorname{ran}(F_{\bar c})).
 \]
 Each \(G_i\) is definable, and by the induction hypothesis,  
 \(\operatorname{ran}(G_i)\) (for \(i=0,1\)) is bounded.  
 Therefore, \(\operatorname{ran}(F)\) is also bounded. $(**)_n$ was established. 
\medbreak
 Now we prove $(\dagger)$. Let $\varphi(x,\bar y)$ be a formula with $|\bar y|=n$ such that 
 
 \[M \models \forall \bar y \in (a,b)^n \; \exists x \; \varphi(x,\bar y).\]

We define a definable function $F:(a,b)^n \to D$ by:
 \[
 F(\bar y)=\begin{cases} 
 \min\{w \in D: \exists x \in (b,w) \varphi(x, \bar y)\} & \text{if $\exists x \in (b,\infty) \varphi(x,\bar y)$,} \\
 \max\{w \in D: \exists x \in (w,b] \varphi(x, \bar y)\} & \text{otherwise}.
 \end{cases}
\]
 By $(**)_n$, $\ran(F)$ is a bounded set.  Choose $c,d \in M$ such that $c< \ran(F) <d$.
 Then, we have $M \models \forall \bar y \in (a,b)^n \exists x \in (c,d) \varphi(x, \bar y)$.
 Thus, $(\dagger)$ holds in $M$. (End of Proof of Claim B).

\medbreak
For simplicity, we work under the assumption that Case 1 and Claim A hold.
Let \( a \in M \), and let \(\lambda\) be the cofinality of the segment \((-\infty,a)^M\).  
Let \( U \) be an ultrafilter on \(\lambda\) that contains all sets of the form \(\lambda \smallsetminus i\) for \( i < \lambda \).  
Choose an increasing sequence \(\{a_i : i < \lambda\}\) that converges to \(a\).  
Define $a^* = [(a_i)_{i<\lambda}] \in M^{(\lambda)}/U$. 
Then \( a^* \) is a new element that is infinitely close to \(a\).  
Hence, by Proposition \ref{prop-bounded} (1), \( M^* = M^{(\lambda)^+}/U \)  is a proper elementary extension of \(M\). 
Moreover, by Theorem \ref{thm-almost-o-mini}(1), \( M^* \) remains almost o-minimal.
\end{proof}

\end{document}